\documentclass[a4paper]{amsart}
\usepackage{amssymb}
\usepackage{amsmath}
\usepackage{graphics,color,url}

\def\CP{{\mathbf C\mathbf P}}
\def\MU{{\mathbf M\mathbf U}}
\def\MSU{{\mathbf M\mathbf S \mathbf U}}

\newcommand{\SU}{\mathop{\mathrm{SU}}}

\numberwithin{equation}{section}

\newtheorem{theorem}{Theorem}[section]
\newtheorem{lemma}{Lemma}

\theoremstyle{definition}

\newtheorem{remark}{Remark}

\begin{document}
	


\title[ ]{Complex cobordism modulo $c_1$-spherical cobordism and related genera}
 
\author[]{M. Bakuradze}

\address{Faculty of exact and natural sciences, A. Razmadze Math. Institute, Iv. Javakhishvili Tbilisi State University, Georgia}

\thanks{The author was supported by EU fellowships for Georgian researchers, 2023 (57655523)}


\email{malkhaz.bakuradze@tsu.ge}

\begin{abstract}
We prove that  the ideal  in complex cobordism ring $\MU^*$  generated by the polynomial generators $S=(x_1, x_k, k\geq 3)$ of $c_1$-spherical cobordism ring $W^*$, 
viewed as elements in $\MU^*$ by forgetful map is prime.  Using the Baas-Sullivan theory of cobordism with singularities we define a commutative complex oriented cohomology theory $\MU^*_S(-)$, complex cobordism modulo $c_1$-spherical cobordism, with the coefficient ring $\MU^*/S$. Then
any $\Sigma\subseteq S$ is also regular in $\MU^*$ and therefore gives a multiplicative complex oriented  cohomology theory $\MU^*_{\Sigma}(-)$.  The generators of $W^*[1/2]$ can be specified in such a way that for $\Sigma=(x_k, k\geq 3)$ the corresponding cohomology is identical to the Abel cohomology, previously constructed in \cite{BUSATO}. Another example corresponding to $\Sigma=(x_k, k\geq 5)$ gives the coefficient ring of the universal Buchstaber formal group law after tensored by $\mathbb{Z}[1/2]$, i.e., is identical to the scalar ring of the Krichever-Hoehn complex elliptic genus \cite{KR}, \cite{H}.
\end{abstract}

\maketitle

\section{Introduction and statements }

The theory of complex cobordism $\MU^*(-)$ and special unitary cobordism  $\MSU^*(-)$ play an important role in cobordism theory \cite{Stong}.
The ring $\MSU^*$ is torsion free after localized away from 2 
\begin{equation}
\label{MSU_*}
\MSU^*[1/2]=\mathbb{Z}[1/2][y_2, y_3,\cdots],\,\,\,\,|y_i|=2i.
\end{equation}
 
With $SU$-structure forgetful inclusion in mind 
$$\MSU^*[1/2]\subset \MU^*[1/2]=\mathbb{Z}[1/2][a_1,a_2,a_3,\cdots],\,\,\,\,|a_i|=2i,$$
the generators $y_i$  
can be treated as elements in $\MU^*[1/2]$. 

In particular $y_i$ is a $SU$-manifold if and only if all Chern numbers of $y_i$ having factor $c_1$ are zero. 
Then we have to check the main Chern number $s_i(y_i)$ for Novikov's criteria  \cite{NOV} (see \eqref{Novikov} below) for the membership of the set of polynomial generators in $\MSU^*[1/2]$.   

In \cite{C-F} Conner and floyd introduced $c_1$-spherical cobordism groups $W^{2n}\subset \MU^{2n}$.  A complex cobordism class belongs to $W^{2n}$ if and only if every Chern number involving $c_1^2$ vanishes.

$W^*$ is not a subring of $\MU^*$. However, with respect to some complex orientation and some $*$ multiplication  
(see Section \ref{Pre}), the ring $W^*$ is polynomial on generators in
every positive even degree except 4, \cite{C-P}, \cite{L-C-P}: 
\begin{equation}
	\label{W_*}
	W^*=Z[x_1, x_k : k \geq 3 ],\,\, x_1 = \CP_1,\, x_k \in W_{2k}.
\end{equation}
The polynomial generators $x_k$ can be specified by the condition for the main Chern numbers
$s_k(x_k)= d(k) d(k-1)$ for $k \geq 3$, where
\begin{equation*}
	d(k)=
	\begin{cases} p, &\mbox{if $k+1=p^s$ for some prime  $p$,}\\
		1, &\mbox{otherwise}.
	\end{cases}
\end{equation*}

By \cite{C-P} the cohomology theory $W^*(-)$ is complex oriented and (by Proposition 3.12) the ring $W^*[1/2]$ is generated by the coefficients of the corresponding formal group law 
$F_W$. In particular

\begin{align}
\label{free}
 &W^*[1/2]\text{ is a free } \MSU^*[1/2] \text{ module generated by 1 and } [\CP_1].
\end{align}

\bigskip

Our main observation is that the sequence of polynomial generators of $W^*$, $(x_k,  k\geq 3)$ in \eqref{W_*} viewed as elements in $\MU^*$ by forgetful map, generate the same ideal in $\MU^*[1/2]$ as the coefficients $\alpha_{ij}, i,j\geq 2$ of the universal formal group law. 
 
We prove in Section 3 one of our main results

\begin{theorem}
	\label{regular}
	Let  $S=(x_1, x_k,  k\geq 3)$ be a sequence of polynomial generators of $c_1$-spherical cobordism $W^*$. Then
	
	i) $S$ is regular in $\MU^*$. 
	
	ii) any subsequence $\Sigma$ of  $S$ is regular in $\MU^*$.
\end{theorem}

\bigskip

Then the use of Baas-Sullivan theory of cobordism with singularities defines  
a multiplicative cohomology  theory
$\MU^*_{\Sigma}(-)$,
with the scalar ring $\MU^*/(\Sigma)$. 

The obstruction to commutativity in $\MU^*_{\Sigma}(-)$ studied in \cite{MI}. In some cases all obstructions vanish only for dimensional reasons. So is for $\Sigma=S$ in Theorem \ref{regular}.

\medskip

In Section 3  we define $\MU^*_{S}(-)$, the complex cobordism modulo $c_1$-spherical cobordism. 

\begin{theorem}
	\label{h} Let $\Sigma=S$ in Theorem \ref{regular}.
	There is a commutative complex oriented cohomology theory $\MU^*_{S}(-)$,
	with the coefficient ring $\MU^*/S$. 
\end{theorem}



\bigskip

Let  $\CP_i$ be the cobordism class of complex projective spae of dimension $2i$ and
 $$F_U=\sum \alpha_{ij}x^iy^j$$ 
be the universal formal group law of complex cobordism.
One can choose the following polynomial generators $y_2, y_3$ and $y_4$ in $\MSU^*[1/2]$ viewed as elements in $\MU^*[1/2]$  (see Section \ref{y2y3y4})

\begin{align}
	\label{234}
	&y_2=\CP_2-\frac{9}{8}\CP_1^2,	
	&&y_3=-\alpha_{22},
	&&y_4=-\alpha_{23}+\frac{3}{2}\alpha_{22}\CP_1.
\end{align}

\bigskip

We need the following result  on quasi-toric representatives in the $SU$-bordism ring.

{\bf Theorem} (\cite{L-C-P}, Thm 10.8.) There are quasi-toric $SU$-manifolds $M^{2i}$ , $i\geq 5$, with 
$s_i(M^{2i})=d(i)d(i-1)$ 
for odd $i$ and $s_i(M^{2i})=2d(i)d(i-1)$ for even $i$, and they have the minimum possible characteristic numbers $s_i$ and represent polynomial generators of $\MSU^*[1/2]$. 

\medskip

Let us specify the generators in \eqref{MSU_*} as
$$y_i=[M^{2i}], i\geq 5.$$
The motivation to this is that the Krichever complex elliptic genus (see \cite{B-P}, § E.5)
vanishes on these generators.

\medskip

We could also choose $y_i, i\geq 5$ based on the following result

\medskip

{\bf Theorem} (\cite{TO}, Thm 6.1.) The kernel of the complex elliptic genus on $\MSU^*[1/2]$ is equal to the ideal of $SU$-flops.

\bigskip

One interesting example of Theorem \ref{regular} is  the Abel cohomology $h^*_{Ab}$ constructed in \cite{BUSATO}.

\begin{theorem}
	\label{abelian}
	Let $\Sigma=(x_i, i\geq 3 )$, where $x_i=y_i$ is specified above.Then   $\MU^*_{\Sigma}(-)[1/2]$ is identical to the Abel cohomology $h^*_{Ab}[1/2]$. 
\end{theorem}

Another example of Theorem \ref{regular} is related to the Buchstaber formal group law $F_B$ \cite{B2}, \cite{B4}. After tensored with
 rationals the corresponding classifying map $f_B$ is identical \cite{BV} to Krichver-Hoehn complex elliptic genus \cite{KR}, \cite{H}. The scalar ring  $\Lambda_B$ of  $F_B$ is calculated in \cite{BU-U}. In particular, it has only 2-torsion and the quotient 
$$\Lambda_{\mathcal{B}}=\Lambda_{B}/torsion $$ 
is an integral domain. We will consider this example in Section 3 and prove the following

\begin{theorem} 
	\label{example2}
	Let $\Sigma=(x_i, i\geq 5 )$, where $x_i=y_i$ is specified above. Let $\Lambda_B$ be the scalar ring of the universal Buchstaber formal group law. Then cohomology theory $\MU^*_{\Sigma}(-)[1/2]$ is multiplicative, commutative, complex oriented and has the coefficient ring identical to
	$\MU^*[1/2]/\Sigma:=\Lambda_{B}[1/2]$.	
\end{theorem}

\bigskip

In Section 4 we discuss a genus on $c_1$-spherical cobordism $W^*$ 
\begin{equation}
\label{phi_W }	
\phi_{W}: W^*[1/2] \to \mathbb[Z/2][x_1,x_3,x_4].
\end{equation}

\bigskip

The following questions remains open. 

What about integral versions of Theorem \ref{abelian} and Theorem \ref{example2}? 

Is the cohomology theory  $\MU^*/(x_i, i\geq 5)$ commutative? 
The same question arises for other examples of Theorem \ref{regular} ii).

\bigskip

\section{Preliminaries}
\label{Pre}

\bigskip

\subsection{Krichever complex elliptic genus}

The Krichever complex elliptic genus (also called Krichever-H\"ohn genus) has been defined in \cite{KR} and \cite{H}.  Krichever wrote down its characteristic power series $Q(x)$ using Baker-Akhiezer function. 
In \cite{H}, H\"ohn defined four variable elliptic genus 
\begin{equation}
	\label{KH}
	\phi_{KH}:\MU^*\to \mathbb{Q}[p_1,p_2,p_3,p_4],\,\,\,|p_i=2i|.
\end{equation}
determined by the following property: if one denotes by $f$ the exponent of the corresponding formal group $F_{KH}$, then the series
$$
h(x):=\frac{f'(x)}{f(x)}
$$
satisfies the differential equation
\begin{equation}
	\label{Hoehn}
	(h'(x))^2=S(h(x)),
\end{equation}
where $S$ is the generic monic polynomial of degree 4:
$$
S(t)=p_4+p_3t+p_2t^2+p_1t^3+t^4.
$$

\bigskip

By definition \cite{BU-BU} the universal Krichever formal group law $F_{K}$ is of the form 
\begin{equation}
	\label{eq:F-Kr}
	F_K(x,y)=xw(y)+yw(x)-w'(0)xy+
	\frac{w(x)\beta(x)-w(y)\beta(y)}{xw(y)-yw(x)}x^2y^2,
\end{equation}
where 
$$w=\frac{\partial F_K(x,y)}{\partial x}(x,0)$$
is the invariant differential form of $F_K$ and 
$\beta(x)=\frac{w'(x)-w'(0)}{2x}.$

\medskip

The universal Krichever genus on $\phi_K $ on $\MU^*$ 
is the homomorphism classifying the formal group law $F_K$.

\begin{lemma}(\cite{BV}, Thm 2.3)
\label{K-KH}
 The exponent of $F_K$ satisfies \eqref{eq:F-Kr},  i.e., $\phi_K$ is a specialization of $\phi_{KH}$.
\end{lemma}

\medskip

The universal  Buchstaber formal group law $F_B$ is defined \cite{B1} as a universal example among formal group laws of the form 
$$
\frac{x^2A(y)-y^2A(x)}{xB(y)-yB(x)},
$$ 
where $A(0)=B(0)=1$.

If $A'(0)=B'(0)$ then $B(x)$ is identical to the image of $\omega(x)$ under the  classifying map $f_B$.

\medskip

One can suggest kernel of $\phi_K$ directly in terms of the universal formal group law $F_U$ as follows \cite{B},\cite{B1}

\begin{lemma}
Let $A(x,y)=\sum A_{ij}x^iy^j$ be defined by
$$A(x,y)=F_U(x,y)(xw(y)-yw(x)), \,\, \text{\,\, where \,\,\,}
w(x)=\frac{\partial F_U(x,y)}{\partial x}(x,0) $$
is the invariant differential form of $F_U$. Then $F_k$ is identical to $F_B$ and the classifying map $f_B=f_K$ is the obvious quotient map  
$$\MU^*\to \MU^*/(A_{ij}, i,j\geq 3).$$
\end{lemma}

\begin{remark}
\label{Lambda generators}
As shown in \cite{BU-U}, all torsion elements of $\Lambda_B$, the scalar ring of $F_B$ have order two. Furthermore,
there is a set of multiplicative generators for $\Lambda_B$ with one generator in every
degree of the form $p^k$, where $p$ is a prime and 
$k \geq 0$, and one generator in every degree 
$2^n-2$, where $n\geq 3$; the latter generators generate the
torsion subring.
\end{remark}

\bigskip

\subsection{$c_1$-spherical cobordism}

in \cite{C-F} Conner and floyd introduced the groups $W^{2n}\subset \MU^{2n}$.  A complex cobordism class belongs to $W^{2n}$ if and only if every Chern number involving $c_1^2$ vanishes.

Recall the idempotent in \cite{BU-1}, \cite{C-P} 
$$\pi_0:\MU^*\to \MU^* : \,\,\,  \pi_0=1 + \sum_{k \ge 2} \alpha_{1k} \partial_k$$
and the projection $\pi_0:\MU^*\to W^*=Im \pi_0$.  

 By \cite{C-P} Proposition 2.15  the multiplication $*$ in $W^*$ is given by
$$
a*b = \pi_0(a\cdot b)=ab + 2[V]\partial a \partial b, 
$$
where $[V]=\alpha_{12}\in\MU^4$ is the cobordism class $\CP_1^2 - \CP_2$.

Then $W^*$ is a subgroup, not a subring in $\MU^*$. With respect to the multiplication $*$ 
the ring $W^*$ is polynomial on generators in
every positive even degree except 4 as in \eqref{W_*}.

 \bigskip
 
 \subsection{The Abel formal group law}
 
By definition \cite{B-KH} the universal Abel formal group law formal group is the universal example of the formal group laws of the form 
$$
F_{Ab}=x+y+\alpha_{11}xy+\sum_{i\geq 2} \alpha_{1i}(xy^i+x^iy).
$$

By Quillen \cite{Q} and Novikov \cite{NOV} the scalar ring of the universal Lazard formal group law is identical to $\MU^*$, the ring of complex cobordism. So to put it another way $F_{Ab}$ is classified from the universal formal group law $F_{U}$ over $\MU^*$
$$
F_U=x+y+\sum \alpha_{ij}x^iy^j
$$ 
by the quotient of $\MU^*$ modulo the ideal $I_{Ab}$ generated by all coefficients $\alpha_{ij}\in \MU^*$, but only $\alpha_{1i}=\alpha_{i1}$
\begin{equation}
\label{I_{Ab}}
I_{Ab}=(\alpha_{ij}: i,j\geq 2).
\end{equation}

Buchstaber named the group law $F_{Ab}$ after Abel, motivated by Abel's article: "Abel N. H. Methode generale pour trouver des fonctions d'une seule quantite variable
lorsqu'une propriete de ces fonctions est exprimee par une equation entre deux variables oeuvres completes, I. Christiania, 1881. p. 1-10."

In that article, Abel gave a general solution of a functional equation for the exponent series of $F_{Ab}$.

\medskip

\begin{remark}
\label{LambdaAB generators}
For our purposes, let us recall from \cite{B-KH} that the scalar ring $F_{Ab}$, i.e., 
$$\Lambda_{Ab}=\MU^*/I_{Ab}$$
is non polynomial infinitely generated ring. There is a set of multiplicative generators for $\Lambda_{Ab}\otimes \mathbb{F}_2$ with one generator in every
complex degree $2^k$, $k\geq 0$. 
 
$\Lambda_{Ab}$ is an integral domain. After tensored with rational numbers it is a subring of the polynomial ring in two parameters of degree 2 and 4, i.e., 
\begin{equation}
 	\label{Lambda-rationally}
 	\Lambda_{Ab}\otimes \mathbb{Q}\subset \mathbb{Q}[\alpha, \beta], \,\,\,|\alpha|=2, |\beta|=4.
 \end{equation}
 \end{remark}
 
 \bigskip

 in \cite{BUSATO} proved that $F_{Ab}$ can be realized by commutative complex oriented cohomology theory $h^*_{Ab}$ also called after Abel.
 
 \medskip
 
 We will also need 
 
 \begin{lemma}
 	\label{specialization}
 	The classifying map $f_{Ab}:\MU^*\to\Lambda_{Ab}$ followed by  $\Lambda_{Ab}\otimes \mathbb{Q}$ in \label{Lambda-rationally} is a specialization of the Krichever-Hoehn complex elliptic genus $\phi_{KH}$. 	
 \end{lemma}
 
 \begin{proof}
 	A direct proof of this fact could include a verification of the defining property \eqref{eq:F-Kr} of the Krichever-H\"ohn genus in case of $F_{Ab}$. 
 Leaving this to the reader, we derive the lemma from the fact that $F_{Ab}$ is obviously a specialization of $F_K$ as $w(x)=1+\sum_{\geq 1}\alpha_{1i}x^i$ and $\phi_K$ is a specialization of $\phi_{KH}$ by Lemma \ref{K-KH}.
 	
 \end{proof}

 \medskip
 
 \subsection{Some auxiliary combinatorial definitions}

Before we continue about the construction of the Abel cohomology, we need to recall some combinatorial facts.

By Euclid's algorithm for the natural numbers $m_1,m_2, \cdots, m_k$  one can find integers $\lambda_1, \lambda_2, \cdots \lambda_k $ such that 
\begin{equation}
	\label{lambda}
	\lambda_1m_1+\lambda_2m_2+\cdots  +\lambda_km_k=gcd(m_1,m_2,\cdots ,m_k).
\end{equation}

Let 
\begin{equation}
	\label{d}
	d(m)=gcd \bigg \{ \binom{m+1}{1},\binom{m+1}{2},\cdots ,\binom{m+1}{m-1}   \big|\,\, m\geq 1 \bigg \}.
\end{equation}

By \cite{KU} one has 
\begin{equation*}
	d(m)=
	\begin{cases} p, &\mbox{if $m+1=p^s$ for some prime  $p$,}\\
		1, &\mbox{otherwise}
	\end{cases}
\end{equation*}
therefore the elements
\begin{equation}
	\label{e_n}
	e_{m}=\lambda_1\alpha_{1\,m}+\lambda_2\alpha_{2\,m-1}+\cdots +\lambda_m\alpha_{m\,1}	
\end{equation}  
may be chosen as the polynomial generators in $\MU^*$.

Let $m\geq 4$ and let $\lambda_2,\cdots ,\lambda_{m-2}$ are such integers that

\begin{equation}
	\label{d_2}
	d_2(m):=\sum_{i=2}^{m-2}\lambda_i\binom{m+1}{i}=gcd \bigg \{ \binom{m+1}{2},\cdots ,\binom{m+1}{m-2}  \bigg \}.
\end{equation}

Then by \cite{BU-U} Lemma 9.7 one has for $m\geq 3$

\begin{align} 
	\label{d_2d_1} 
	&d_2(m)=d(m)d(m-1).	
\end{align}

Let us apply Euclid's algorithm for the Chern numbers $$s_{m-1}(\alpha_{i,m-i})=\binom{m+1}{i}$$
in \eqref{d_2} and let
$$z_{k}=\sum_{i=2}^{k-1}\lambda_i\alpha_{i\, k+1-i},\,\,\,k\geq 3. $$   

Let us recall from \cite{BUSATO} the constriction of abelian cohomology $h_{Ab}^*(-)$.   

\begin{lemma}
\label{I=I_{Ab}}
Let $I_{Ab}$ be as in \eqref{I_{Ab}} and  $I_{Ab}(l)\subset I_{Ab}$ be generated by  those elements of $\MU^*$ whose degree is less or equal $2l$. Let $I=(z_3, \cdots)$ and $I(l)=(z_3,\cdots, z_l )$. Then 

i) $I_{Ab}=I$;

ii) $I_{Ab}(l)=I(l)$.
\end{lemma}

Let's put $\MU^*=\mathbb{Z}[e_1,e_2, \cdots]$ for $e_m$ in \eqref{e_n} and let
$$A_l=\mathbb{Z}[e_1,e_2,\cdots , e_l] \text{ and } 
A^{l+1}=\mathbb{Z}[e_{l+1}, e_{l+2}, \cdots],
$$ 
i.e., 
$$\MU^*=A_l\otimes A^{l+1}.$$   

The preimage of $I(l)$ by obvious inclusion defines the ideal of $A_l$ denoted by same symbol so that 
\begin{equation}
\label{l}
\MU^*/I(l)=A_l/I(l)\otimes A^{l+1}.
\end{equation}

\begin{remark}
\label{I(l)}
We note here that by definition $I(l)$ is generated by those polynomials in 
$\mathbb{Z}[e_1,e_2, \cdots , e_l]$ that are in the
kernel of the quotient map $\MU^* \to \MU_*/I_{Ab}$.
\end{remark}

\medskip

The construction of abelian cohomology is based on the following key result \cite{B-KH}, \cite{BUSATO}.
 
\begin{lemma}
\label{regular-J}
	The ideals $I$ and  $I(l)$ are prime, 
	or equivalently   $\MU^*/I$ and $\MU^*/I(l)$ are integral domains.
  \end{lemma}

Using Baas-Sullivan theory the commutative complex oriented cohomology  $h_{Ab}^*(-)$ with coefficient ring $\MU^*/I$  was defined in \cite{BUSATO}. 

\medskip

\begin{remark}
\label{remark regular}
We recall that for a commutative ring $R$, a regular sequence  is a sequence $(r_1, \cdots, r_n)$ in $R$ such that $r_i$ is not a zero-divisor on $R/(r_1, \cdots, r_{i-1}$, for $i=1, \cdots,  n$. 
	
If $R$ is a graded ring and the members of a regular sequence are homogeneous of positive degree, then any permutation of a regular sequence is  regular \cite {BRUNS}. Therefore we can regard our regular sequences as regular sets.
\end{remark}

%
%

\medskip

\section{Proofs}

\begin{lemma} 
\label{main lemma}

Let $J_W$ be the ideal in $\MU^*$ generated by $y_2=w_1*w_1=8\CP_2-9\CP_1^2$ and polynomial generators  $(x_i, i \geq 3)$ of $W^*$  as elements in $\MU_*$ by group inclusion and $J_W(l)$ be $(y_2,x_3, \cdots , x_l)$. 
Let $I$ and $I(l)$ be the ideals in  
Lemma \ref{I=I_{Ab}}, 
$\tilde{I}=(y_2,I)$ and $\tilde{I}(l)=(y_2,I(l))$.  Then

i) $J_W(l)=\tilde{I}(l)$, $\forall l\geq 3$;

ii) The sequence $y_2,x_3, \cdots $ is regular in $\MU^*$. 	
\end{lemma}

\begin{proof}
i) First let us prove $\tilde{J}_W\subset \tilde{I}$.  We have $I=I_{AB}$ by Lemma \ref{I=I_{Ab}}. 
Also $\Lambda_{Ab}=\MU^*/I$ is an integral domain by Lemma \ref{regular-J}. 
Moreover $\MU^*/I \otimes \mathbb{Q}$ is polynomial ring in two parameters $\CP_1,\CP_2$ by \eqref{Lambda-rationally}. 
Therefore it suffices to check $x_k \in \tilde{I}$ viewed as elements in $\MU^*[1/2]$. 

For the latter we need two facts. 
 
By $SU$ representatives in Introduction  $J_{SU} \subset I$, where  $J_{SU}$ is the ideal in $\MU^*[1/2]$ generated by 
$y_5, y_6, \cdots $, special unitary flops in $\MSU^*[1/2]$ of dimensions 
$\geq 10 $. These generators vanish under the Krichever genus and therefore vanish under its specialization, the classifying map of the Abel formal group law.  Then $y_3,y_4 \in I$ by 
 \eqref{234}  therefore  $(y_2,y_3,y_4,J_{SU}) \subset \tilde{I}.$
 
 Then by \cite{C-P} $W^*[1/2]$ is a free $\MSU^*[1/2]$ module generated by $1$ and $\CP_1$. So after inverting 2 each generator $x_k$, $k\leq l$, of $J_W(l)$ as element in $\MU^*[1/2]$ 
 is in kernel of the quotient map  to the integral domain 
 $\MU^*/(I(l),x_2)$. 
 Thus $x_k \in \tilde{I}(l)$ and $J_W \subset \tilde{I}$ as the ideals in $\MU^*$. Here we use Remark \ref{remark regular} for $\tilde{I}(l)$.  
 
 Now let's prove $I(l)\subset J_W(l)$, $\forall l\geq 2$ by induction on $l$. For $l=2$ and $l=3$ this is true because $y_2=8\CP_2-9\CP_1^2$ and $x_3=y_3=z_3=-\alpha_{22}$. Let $I(l-1)\subset J_W(l-1)$ by induction hypothesis. 
 To prove  $z_l\in (J_W(l-1),x_l)$ recall that  the main Chern number $s_{l}(z_l-x_l)=0$ by the definitions of $z_l$, $x_l$, i.e., \eqref{W_*} and \eqref{d_2d_1}. This means that $z_l-x_l=P(e_1,e_2,\cdots , e_{l-1})$ as decomposable element in $\MU^*$. Here $e_i$ are as in \eqref{e_n}. On the other hand $x_l\in J_W(l)\subset \tilde{I}(l)$ hence  $P=x_l-z_l$ is in the kernel of the quotient map $\MU^*\to \MU^*/\tilde{I}(l-1)$, i.e. $P\in \tilde{I}(l-1)=J_W(l-1)$. The induction hypothesis and Remark \ref{I(l)} completes the proof.  

\medskip

ii)  By \cite{BUSATO} again $I$ is prime ideal, therefore so is $\tilde{I}=J_W$ by Remark \ref{I(l)}. 



\end{proof}

\subsection{Proof of Theorem \ref{regular}. }

i) As $x_3,\cdots,x_l\subset J_W(l)$  by Lemma \ref{main lemma} ii) and Remark \ref{remark regular} the ring $R=\MU^*/(x_3,\cdots, x_l)$ is an integral domain. In particular,
the multiplication by $x_1=\CP_1$ is mono in $R$.  By Remark \ref{remark regular} again  $(x_1, x_3, \cdots, x_l)$ is regular in $\MU^*$ 
$\forall l$ and one has ii).

\subsection{Proof of Theorem \ref{abelian}}  
We use Lemma \ref{main lemma} and apply the specified generators $x_3,\cdots$ in $W^*[1/2]$ such that 
$(x_3,\cdots)=(z_3,\cdots )$ in $\MU^*[1/2]$.

\qed

\subsection{Proof of Theorem \ref{h}.}

We have to prove that all obstructions to commutativity vanish. This follows by  dimension argument and definition of obstruction elements in \cite{MI}. In particular, the obstruction elements to commutativity 
$\tilde{\mathcal{P}}_2(x_i)$
defined by certain cohomology operation \cite{MI} Theorem 3.4
$$
\tilde{\mathcal{P}}_2:\MU^{-2n}\to \MU^{-4n-2}\otimes \mathbb{F}_2.
$$ 
So that all obstruction elements 
are of odd complex dimension $-2i-1$. Therefore they vanish in  the scalar ring $\MU^*/(x_1,x_3,\cdots)\otimes \mathbb{F}_2$ which is evenly generated  as a quotient of $\Lambda_{Ab}$ by Remark \ref{LambdaAB generators}.

Alternatively, we can use Busato's long proof  \cite{BUSATO} of commutativity of $h^*_{Ab}(-)$. The proof is that $\tilde{\mathcal{P}}_2(\alpha_{ij}) $ projects to $0$ in $\Lambda_{Ab}\otimes \mathbb{F}_2$, for $\alpha_{ij} , \,\,i,j\geq 2$. In our case, we have to consider one more element $\alpha_{11}=x_1$. By \cite{MI} Theorem 3.4 we have $\tilde{\mathcal{P}}_2(\alpha_{11})=\alpha_{22}=x_3 $ modulo $\alpha_{11}$.
So that 
the obstruction element $\tilde{\mathcal{P}}_2(x_1)$ vanishes.

\subsection{Proof of Theorem \ref{example2}.}

One can use the fact that $\MSU^*[1/2]$ is a subring in $W^*[1/2]$ and follow  \cite{B0} where we proved the statement for generators $(y_i,i\geq 5)$ in $\MSU^*[1/2]$.  The scalar ring  $\Lambda_B$ of $F_B$ and its classifying map $f_B$ is calculated in \cite{BU-U} (see also Remark \ref{Lambda generators}). In particular,  $\Lambda_B$ has the ideal of elements of order two generated by  $f_B(e_{2^n-2})$, $n\geq 3$.
Here $e_i$ are as in \eqref{e_n}.
The quotient by torsion
$$\Lambda_{\mathcal{B}}=\Lambda_{B}/(f_B(e_{2^n-2}),\,\, n\geq 3).$$ 
is an integral domain. 

\qed

\begin{remark}
Using the arguments in the proof of Theorem \ref{h} one can prove that for 
$\Sigma$ containing all generators $x_n$ of $W^*$ with odd indices $n=2i-1\geq 1$, the cohomology theory $h^*_{\Sigma}$ is commutative.
\end{remark}

\subsection{Proof of  \eqref{234}}
\label{y2y3y4}
We have to check that $y_2,y_3,y_4$ are represented by $\SU$-manifolds i.e., all Chern numbers of $y_i$ having factor $c_1$ are zero. Then we have to check the main Chern number $s_i(y_i)$ for Novikov's criteria  \cite{NOV} for the membership of the set of polynomial generators in  
$\MSU^*[1/2]$.

In particular, $\SU$ manifold $x_n$ is a polynomial
generator if and only if 
\begin{align}
	\label{Novikov}
	s_n(y_n)=
	\begin{cases} \pm 2^kp &\mbox{if } n=p^l,\,\,\,p\,\,\, \mbox{is odd prime, }\\
		\pm 2^kp &\mbox{if } n+1=p^l,\,\,\, p\,\,\,\mbox{is odd prime, } \\
		\pm 2^k& \mbox{otherwise}.
	\end{cases}
\end{align}

Rewrite the elements $y_i$ in $\MU^*$ as follows

 
\begin{align*}
&y_2=\CP_2-\frac{9}{8}\CP_1^2;\\
&2y_3=-3\CP_3+8\CP_1\CP_2-5\CP_1^3;\\
&y_4=-2\CP_1^4+7\CP_1^2\CP_2-3\CP_2^2-4\CP_1\CP_3+2\CP_4.\\
\end{align*}

The proof follows from the following calculations for the main Chern numbers.

\begin{align*}
&c_1c_1[\CP_1^2]=8, &&c_1c_1[\CP_2]=9, &&&c_2[\CP_1^2]=4, &&&&& c_2[\CP_2]=3.\\ 
\end{align*}
This implies that that $c_1c_1(y_2)=0$. There are no more Chern numbers involving $c_1$  and  $s_2(y_2)=3$.  Therefore $y_2$ forms a generator of $\MSU^4[1/2]$.

\medskip

For $y_3$:
\begin{align*}
	& X             &c_3(X)   && c_1c_2(X)     &&& c_1c_1c_1(X)\\
	&\CP_3           &4        && 24            &&& 64  \\
	&\CP_1\CP_2     &6        && 24            &&& 54    \\
	&\CP_1^3        &8        && 24            &&& 48.\\
\end{align*}

It follows all Chern numbers of $y_3$  involving $c_1$ are zero. 
Then $y_3$ forms a generator in $\MSU^6[1/2]$ as $s_3(y_3)=-2\cdot 3$ satisfies the  Novikov criteria.

Similarly for $y_4$:  we have $s_4(y_4)=2\cdot 5$  and 

\begin{align*}
	&X          &c_1c_1c_1c_1(X)  &&c_1c_1c_2(X) &&&c_1c_3(X)    \\
	&\CP_1^4     &384             &&192          &&&64              \\
	&\CP_1^2\CP_2 &432            && 204         &&&60           \\  
	& \CP_2^2     &486            && 216         &&&54           \\
	&\CP_1\CP_3   &512            && 224         &&&56            \\
	&\CP_4       &625             && 250         &&&50.              \\
\end{align*}

\qed

\section{$c_1$-spherical cobordism modulo quasi-toric manifolds}

 Recall again from \cite{TO} that in special unitary cobordism $\MSU^*[1/2]$ one can construct polynomial generators $y_n,\,\,\,n\geq 5$ using Euclid's algorithm and $SU$-flops or equivalently generators specified in Introduction.
 
 The restriction of the Krichever complex elliptic genus
 $$
 \MU^*\to \mathbb{Q}[p_1,p_2,p_3,p_4]
 $$
 on $\MSU^*[1/2]$ gives the genus
 $$
 \MSU^*[1/2]\to \mathbb{Z}[1/2][y_2,y_3,y_4],
 $$
 as it is zero on all $y_i, i\geq 5$.

As $W^*[1/2]$ is  a free $\MSU^*[1/2]$ module generated by 1 and $x_1=\CP_1$  one has the following

\begin{theorem}
$SU$-flops of dimension $\geq 10$ as elements in the ring $W^*[1/2]$ with $*$ product generate the kernel ideal of the genus
$$
\phi_W:W^*[1/2]\to\mathbb{Z}[1/2][x_1,x_3,x_4].
$$
\end{theorem}



\begin{thebibliography}{99}



	
\bibitem{BA} N. Baas,  On bordism theory of manifolds with singularities,\emph{ Math. Scand.}, {\bf 33}(1973), 279--302.

\bibitem{B0} M. Bakuradze, {\emph {Polynomial generators of $\MSU^*[1/2]$ related to classifying maps of certain formal group laws}}, Homology, Homotopy and Applications, to appear.


\bibitem{BV} M. Bakuradze, V.V. Vershinin {\emph { On addition theorems related to elliptic integrals}}, Proc. Steklov Math. Inst., {\bf 305}(2019), 22-32.



\bibitem{B} M. Bakuradze, {\emph {On the Buchstaber formal group law and some related genera}}, Proc. Steklov Math. Inst., 286(2014), 7-21.


\bibitem{B1} M. Bakuradze, {\emph {Formal group laws by Buchstaber, Krichever and Nadiradze coincide}}, Russian Math. Surveys, {\bf 68:3} (2013), 571–573.

\bibitem{B2} M. Bakuradze, {\emph {Computing the Krichever genus}}, J.  Homotopy Relat. Struct, {\bf 9, 1}(2014), 85-93.

\bibitem{B4} M. Bakuradze, {\emph {Cohomological realization of the Buchstaber
		formal group law}}, Russian Math. Surv., {\bf 77:5} (2022), 943-945. 


\bibitem{BU-BU} V. M. Buchstaber, E. Yu. Netay, 
{\emph {$CP(2)$-multiplicative Hirzebruch genera and elliptic cohomology}}, 
Russian Math. Surveys, {\bf 69:4} (2014), 757-759.


\bibitem{BU-1} V. M. Buchstaber, {\emph {Projectors in unitary cobordisms that are related to SU-theory}}, Uspekhi Mat. Nauk, 
{\bf 27:6(168)}(1972), 231-232.





\bibitem{B-P} V. M. Buchstaber and T. E. Panov, {\emph {Toric topology}}, Math. Surveys Monogr., {\bf vol. 204}, Amer. Math. Soc., Providence, RI 2015, xiv+518 pp.

\bibitem{BU-U} V. Buchstaber, K. Ustinov, {\emph {Coefficient rings of Buchstaber formal group laws}}, Math. Notices, {\bf 206}:11(2015), 19-60.

\bibitem{B-KH} Bukhshtaber V.M., Kholodov A.N., {\emph {Formal groups, functional equations, and generalized cohomology theories, (English. Russian original)}}, Math. USSR Sbornik {\bf 69:1} (1991), 77–97.

\bibitem{BRUNS} W. Bruns, J. Herzog, {\emph {Cohen-Macaulay rings}}, Cambridge University Press, (1993).


\bibitem{BUSATO} Ph. Busato, {\emph {Realization of Abel’s universal formal group law}}, Math. Z., {\bf 239}(2002), 527–561.



\bibitem{C-P} G. Chernykh, T. Panov, {\emph {$SU$-linear operations in complex cobordism and the $c_1$-spherical bordism theory}}, Izvestiya Ross. Akad. Nauk Ser. Mat. {\bf 87} (2023), no. 4.



\bibitem{C-F} P.E Conner, E.E. Floyd {\emph {Torsion in SU-bordism}}, Mem. Amer. Math. Soc. {\bf 60} (1966).




\bibitem{H} G. H\"ohn,  {\emph {Komplexe elliptische Geschlechter und $S^1$-\"aquivariante Kobordismustheorie}}, 2004, https://arxiv.org/abs/math/0405232

\bibitem{KR} I. Krichever, {\emph {Generalized elliptic genera and Baker-Akhiezer functions}}, Math. Notes, {\bf  47}(1990),  132–142. 

\bibitem{KU} E. E. Kummer, {\emph {Uber die Erganzungssatze zu den allgemeinen Reciprocitatsgesetzen}}, J. Reine Angew. Math., {\bf 44 }(1852), 93–146.


\bibitem{L-C-P} I. Yu. Limonchenko, T. E. Panov, and G. S. Chernykh {\emph {SU-bordism: structure results and geometric representatives}}, Russian Math. Surv., {\bf 74:3}(2019), 461–524.


\bibitem{MI} O. K. Mironov, {\emph {Multiplications in cobordism theories with singularities, and Steenrod – tom Dieck operations}},  Math. USSR
Izvestija {\bf 13}(1979), 89–106; 




\bibitem{NOV} S. P. Novikov, {\emph {Homotopy properties of Thom complexes (Russian)}}, Mat. Sb. {\bf 57}(1962), 407–442.

\bibitem{NOV} S. P. Novikov,
The methods of algebraic topology from the viewpoint of cobordism theory (in Russian), \emph{Izv. Akad. Nauk SSSR Ser. Mat.} \textbf{31} (1967), no. 4, 855--951;
translation in \emph{Math. USSR-Izv.} \textbf{1} (1967), no. 4, 827--913.





\bibitem{Q}
D. Quillen,
On the formal group laws of unoriented and complex cobordism theory,
\emph{Bull. Amer. Math. Soc.} \textbf{75} (1969), 1293--1298.







\bibitem{Stong} R. E. Stong, {\emph {Notes On Cobordism Theorey}}, Princeton University Press and University of Tokyo Press, (1968)






\bibitem{TO} B. Totaro, {\emph {Chern numbers for singular varieties and elliptic homology}}, Annals of Mathematics, {\bf 151}(2000), 757–792.









\end{thebibliography}
\end{document}